\documentclass{amsart}

\usepackage{graphicx, layout, appendix, subcaption}
\usepackage{fancyhdr}
\usepackage{amssymb, amsthm, amsmath, amscd, amsfonts,
mathrsfs, latexsym, mathtools}
\usepackage{placeins} % preamble
\usepackage{enumitem}
\usepackage{setspace}
\usepackage{changepage}
\usepackage{leftidx}
\usepackage{calc}
\usepackage{verbatim}
\usepackage{hyphenat}
\usepackage[
	setpagesize=false,
	colorlinks=true,
	linkcolor=blue,
	pdfencoding=auto,
	psdextra,
]{hyperref}
\usepackage{cleveref}
\usepackage[final]{microtype}
\usepackage{xcolor}
\hypersetup{
    colorlinks,
    linkcolor={red!50!black},
    citecolor={blue!50!black},
    urlcolor={blue!80!black}
}

\usepackage{ifthen}
\usepackage{amsxtra}
\usepackage{amstext}
\usepackage{stmaryrd}
\usepackage{pifont}
\usepackage[T1]{fontenc}
\usepackage{textcomp}
\usepackage{bbm}
\usepackage{tikz-cd}

\makeatletter
\def\section{\@startsection{section}{1}%
  \z@{.7\linespacing\@plus.2\linespacing}{.5\linespacing}%
  {\normalfont\scshape\centering}}
\def\subsection{\@startsection{subsection}{2}%
  \z@{.5\linespacing\@plus.15\linespacing}{-.5em}%
  {\normalfont\bfseries}}
\def\subsubsection{\@startsection{subsubsection}{3}%
  \z@{.5\linespacing\@plus.15\linespacing}{-.5em}%
  {\normalfont\itshape}}
\def\thm@space@setup{%
  \thm@preskip=.5\baselineskip\@plus.1\baselineskip
                             \@minus.2\baselineskip
  \thm@postskip=\thm@preskip
}
\makeatother
\setlist[enumerate,1]{label=(\arabic*), ref=(\arabic*)}
\setlist[enumerate,3]{label=(\roman*), ref=(\roman*)}

\theoremstyle{plain}
\newtheorem{thm}{Theorem}[section]
\newtheorem*{thm*}{Theorem}

\newtheorem{lem}[thm]{Lemma}

\newtheorem*{lem*}{Lemma}

\newtheorem{prop}[thm]{Proposition}

\newtheorem*{prop*}{Proposition}

\newtheorem*{cor*}{Corollary}

\newtheorem*{claim*}{Claim}

\newtheorem*{conj*}{Conjecture}

\theoremstyle{definition}
\newtheorem{defn}[thm]{Definition}
\newtheorem*{defn*}{Definition}

\newtheorem*{ques*}{Question}

\newtheorem{exa}[thm]{Example}

\newtheorem*{exa*}{Example}

\theoremstyle{remark}
\newtheorem{rmk}[thm]{Remark}

\newtheorem*{rmk*}{Remark}

\numberwithin{figure}{section}
\numberwithin{table}{section}
\numberwithin{equation}{section}

\def \CC {\mathbb{C}}

\def \PP {\mathbb{P}}
\def \QQ {\mathbb{Q}}
\def \RR {\mathbb{R}}

\def \ZZ {\mathbb{Z}}

\def \Ocal {\mathcal{O}}

\def \hbar {\bar{h}}

\DeclareMathOperator{\Aut}{Aut}

\DeclareMathOperator{\Pic}{Pic}
\DeclareMathOperator{\NS}{NS}

\DeclareMathOperator{\Int}{int}     % interior

\DeclareMathOperator{\Nef}{Nef}
\DeclareMathOperator{\Eff}{Eff}

\DeclareMathOperator{\Amp}{Amp}
\DeclareMathOperator{\Pos}{Pos}

\newcommand{\NEbar}{\overline{\mathrm{NE}}}

\usepackage{scalerel}
\usepackage{stackengine,wasysym}

\title[Nef Cones of Hilbert Schemes on Generalized Cayley K3]{Nef Cones of the Hilbert Schemes of Points on Generalized Cayley K3 Surfaces}
\author{Chiwon Yoon}
\address{Department of Mathematical Sciences, KAIST, 291 Daehak-ro, Yuseong-gu, Daejeon, 34141, Republic of Korea
}
\email{dbs7985@kaist.ac.kr}

\subjclass[2020]{ % https://mathscinet.ams.org/mathscinet/msc/pdfs/classifications2020.pdf
    14J28, % K3 surfaces and Enriques surfaces
    14C05% Parametrization (Chow and Hilbert schemes)
}

\begin{document}

\begin{abstract}
We study the nef cones of Hilbert schemes of points on generalized Cayley K3 surfaces $S_a$. For the Hilbert squares $S_a^{[2]}$ with $a=1,2$, we determine the nef and effective cones by combining Beauville involutions with the Bayer--Macr\`i wall description and explicit lattice calculations. In the Cayley case $a=1$, we further construct an explicit rational polyhedral fundamental domain for the action of the automorphism group on the nef effective cone. For arbitrary $a$ and $n\ge\lfloor a^2/4\rfloor+3$, we determine the nef and Mori cones using explicit divisor classes whose nefness follows from the Mukai lattice description, together with curve classes obtained from pencils on smooth curves.
\end{abstract}
\maketitle
%---------------------------------------------------------------------
\section{Introduction}\label{sec:intro}

Hilbert schemes of points on K3 surfaces are classical examples of higher-dimensional irreducible holomorphic symplectic manifolds. Let $S$ be a projective K3 surface and let $n\ge2$. Then $S^{[n]}$ is smooth and projective of dimension $2n$; see Fogarty~\cite{Fog68} and Beauville~\cite{Bea83b}. Its Picard group has the decomposition
\[
\Pic(S^{[n]}) \cong \Pic(S) \oplus \ZZ e,
\]
where $e=\tfrac12B$ and $B$ is the exceptional divisor of the Hilbert--Chow morphism~\cite{Fog73}. The second cohomology carries the Beauville--Bogomolov--Fujiki form~\cite{Bea83b,Fuj87}.

One of the main problems in the study of hyperk\"ahler varieties is to describe the nef cone. For K3 surfaces, Sterk~\cite{Ste85} proved that the nef cone is determined by walls orthogonal to $(-2)$-curves, and that the automorphism group acts on it with a rational polyhedral fundamental domain. In higher dimensions, Markman~\cite{Mar11,Mar13} described the nef and movable cones using monodromy, while Amerik and Verbitsky~\cite{AV15,AV17} developed the theory of monodromy birationally minimal classes. The nef cones of Hilbert schemes of points on K3 surfaces have also been studied by explicit lattice-theoretic and numerical methods. Hassett and Tschinkel~\cite{HT01,HT09} studied the ample and movable cones of holomorphic symplectic fourfolds of K3$^{[2]}$-type. Bayer, Hassett, and Tschinkel~\cite{BHT15} extended this analysis to projective K3$^{[n]}$-type varieties.

Bayer and Macr\`i~\cite{BM14a,BM14b} described the nef, movable, and effective cones of moduli spaces on K3 surfaces in terms of the Mukai lattice. Since $S^{[n]}$ is itself a moduli space of ideal sheaves on $S$, their results apply directly to Hilbert schemes of points. For Hilbert squares, we combine the resulting numerical wall criteria with Beauville involutions and explicit lattice calculations. For higher Hilbert schemes, we use the Bayer--Macr\`i description of the Mori cone.

The Cayley K3 surface is a determinantal quartic surface $S\subset\PP^3$ whose Picard lattice has rank two and intersection form
\[
\begin{pmatrix}
4 & 2 \\
2 & -4
\end{pmatrix}.
\]
It was considered by Cayley~\cite{Cay69} and later studied by Oguiso~\cite{Ogu12}. In the Cayley case, Lee~\cite{Lee25} determined the automorphism group of the Hilbert square $S^{[2]}$.

We consider the following generalization of the Cayley K3 surface. For each integer $a\ge1$, let $S_a$ be a projective K3 surface with $\NS(S_a)=\ZZ h_1\oplus\ZZ h_2$ and intersection form
\[
\begin{pmatrix}
4 & 2a \\
2a & -4
\end{pmatrix}.
\]
We call $S_a$ a generalized Cayley K3 surface. The case $a=1$ recovers the Cayley K3 surface. Lee~\cite{lee2026generalized} studied the automorphism groups of these K3 surfaces using generalized Fibonacci numbers.

Section~\ref{sec:Nef_cone} determines the nef and effective cones of $S_1^{[2]}$ and $S_2^{[2]}$, while Section~\ref{sec:FD} constructs an explicit rational polyhedral fundamental domain for the action of $\Aut(S_1^{[2]})\cong\ZZ_2*\ZZ_2*\ZZ_2$ on $\Nef^e(S_1^{[2]})$. The main result for the Hilbert square of the Cayley K3 surface is the following.
\begin{thm}[= Theorem~\ref{thm:nef_cone_of_S^[2]}]
The walls of the nef cone of $S_1^{[2]}$ are precisely the hyperplanes orthogonal to the classes in $\Aut(S_1^{[2]})\cdot e$. Moreover,
\[
\Eff(S_1^{[2]}) = \operatorname{Cone}\bigl(\Aut(S_1^{[2]})\cdot e\bigr).
\]
\end{thm}
For $S_2^{[2]}$, the nef walls are likewise given by the hyperplanes orthogonal to $\Aut(S_2^{[2]})\cdot e$. As described in Theorem~\ref{thm:nef-cone-a2}, its effective cone is generated by $\Aut(S_2^{[2]})\cdot e$ together with the rational isotropic classes in $\overline{\Pos(S_2^{[2]})}$.

For arbitrary $a$, Section~\ref{sec:higher-2} treats $S_a^{[n]}$ in the range $n\ge\lfloor a^2/4\rfloor+3$. We determine both the nef and Mori cones. The nef cone is the closure of the cone generated by the classes $D_{k,t}$, while the Mori cone is the closure of the cone generated by the Hilbert--Chow curve class $C_0$ together with the curve classes $(E_{k,t})_{[n]}$ obtained from pencils on smooth curves. We establish the nefness of the classes $D_{k,t}$ using the Bayer--Macr\`i description of the Mori cone in terms of the Mukai lattice.
\begin{thm}[= Theorem~\ref{thm:nef-cone-large-n}]
Let $a\ge1$ and $n\ge\lfloor a^2/4\rfloor+3$. Then
\[
\NEbar(S_a^{[n]})
=
\overline{\operatorname{Cone}\bigl\langle
C_0,\,(E_{k,t})_{[n]}
\mid k\in\ZZ,\;0\le t\le a-1
\bigr\rangle}.
\]
Moreover,
\[
\Nef(S_a^{[n]})
=
\overline{\operatorname{Cone}\bigl\langle
D_{k,t}\mid k\in\ZZ,\;0\le t\le a-1
\bigr\rangle}.
\]
\end{thm}
\subsection*{Acknowledgments}
The author would like to thank Yongnam Lee for detailed comments on the draft of this paper, and Wanmin Liu for helpful comments on the references.
This work was supported by the Institute for Basic Science (IBS-R032-D1).

\section{Preliminaries}\label{sec:pre}

\subsection{K3 surfaces and their Hilbert schemes}\label{subsec:K3-Hilbert}

Throughout, all K3 surfaces are projective and defined over $\CC$; we refer to~\cite{Huy16} for standard facts on K3 surfaces and their lattices.

For each integer $a\ge1$, let $L_a$ be the rank-two even hyperbolic lattice with Gram matrix
\[
\begin{pmatrix}
4 & 2a \\
2a & -4
\end{pmatrix}.
\]
By Nikulin's theorem on primitive embeddings~\cite{Nik80}, $L_a$ admits a primitive embedding into the K3 lattice $\Lambda_{\mathrm{K3}}\cong U^{\oplus3}\oplus E_8(-1)^{\oplus2}$. By the surjectivity of the period map, there exists a projective K3 surface $S_a$ with $\NS(S_a)\cong L_a$. We fix a basis $(h_1,h_2)$ with the Gram matrix above such that $h_1$ is ample.

Every square in $L_a$ is divisible by $4$, so $L_a$ contains no $(-2)$-classes. Hence the ample cone of $S_a$ is the component of the positive cone containing $h_1$. The form $x^2+axy-y^2$ has discriminant $a^2+4$, which is not a perfect square for $a\ge1$, so $L_a$ contains no nonzero isotropic classes.

\begin{defn}[Beauville--Bogomolov--Fujiki form {\cite{Bea83b,Fuj87}}]\label{def:BBF}
Let $X$ be an irreducible holomorphic symplectic manifold of complex dimension $2n$. There exists a primitive integral quadratic form
\[
q_X\colon H^2(X,\ZZ)\longrightarrow\ZZ
\]
of signature $(3,b_2(X)-3)$, called the Beauville--Bogomolov--Fujiki (BBF) form. Its associated bilinear form is defined by
\[
(D,D')=\tfrac12\bigl(q_X(D+D')-q_X(D)-q_X(D')\bigr),
\qquad D,D'\in H^2(X,\ZZ).
\]
\end{defn}

Let $n\ge2$ and let $X=S^{[n]}$ be the Hilbert scheme of points on a K3 surface $S$. With respect to the BBF form, there is an orthogonal decomposition
\[
H^2(X,\ZZ) \cong H^2(S,\ZZ)\oplus\ZZ e,
\]
where $e=\tfrac12B$ and $B$ is the exceptional divisor of the Hilbert--Chow morphism $S^{[n]}\to S^{(n)}$. For a divisor $D$ on $S$, we write $D^{[n]}$ for the pullback to $S^{[n]}$ of the divisor class on $S^{(n)}$ induced by $D$. We also identify $N^1(S)$ with its image in $N^1(S^{[n]})$ when no confusion can arise. We use $\NS(X)$ for the integral N\'eron--Severi lattice and set $N^1(X)_\QQ:=\NS(X)\otimes\QQ$ and $N^1(X)_\RR:=\NS(X)\otimes\RR$.

The BBF form restricts to the intersection form on $H^2(S,\ZZ)$ and satisfies $q_X(e)=-2(n-1)$. We write $q=q_X$ when $X$ is clear. We use the BBF pairing to identify $N^1(X)_\QQ$ with $N_1(X)_\QQ$, and similarly over $\RR$. In particular, for $X=S_a^{[n]}$ and the basis $(h_1,e,h_2)$,
\[
q(xh_1+ye+zh_2) = 4x^2+4axz-4z^2-2(n-1)y^2.
\]
For a hyperk\"ahler fourfold $X$, the polarized Fujiki relation takes the form $D\cdot H^3=c_Xq_X(H)(D,H)$, where $c_X>0$.

Throughout, ${}^\perp$ denotes orthogonality with respect to the relevant pairing: the BBF form on $H^2(X,\RR)$ and $N^1(X)_\RR$, and the Mukai pairing on the Mukai lattice. Via the BBF identification $N^1(X)_\RR\cong N_1(X)_\RR$, we use the same notation for the corresponding annihilators with respect to the divisor--curve pairing.

We denote by $C_0$ the class of a line in a fiber of the Hilbert--Chow morphism. The relative Mori cone is $\NEbar(X/S^{(n)})=\RR_{\ge0}C_0$. A choice of a $g^1_n$ on a smooth curve $C\subset S$ determines a curve $C_{[n]}\subset S^{[n]}$. We use the same notation $C_{[n]}$ for its numerical class in $N_1(S^{[n]})_\RR$. The corresponding intersection numbers are as follows:
\[
\begin{array}{c|cc}
 & D^{[n]} & B \\
\hline
C_{[n]} & C\cdot D & 2g(C)-2+2n \\
C_0 & 0 & -2
\end{array}.
\]

\subsection{Cones and the cone conjecture}\label{subsec:Cone_conjecture}

For a projective irreducible holomorphic symplectic manifold $X$, let $\Pos(X)$ be the connected component of
\[
\{x\in N^1(X)_\RR\mid q_X(x)>0\}
\]
containing $\Amp(X)$. The restriction of the BBF form to $N^1(X)_\RR$ has signature $(1,\rho(X)-1)$. The closure $\overline{\Pos(X)}$ is self-dual with respect to the BBF form. In particular, if $x\in\Pos(X)$ and $0\ne y\in\overline{\Pos(X)}$, then $(x,y)>0$.

Let $\Eff(X)$ denote the cone generated by effective divisor classes, and let $\overline{\Eff}(X)$ denote its closure, the pseudo-effective cone. Set
\[
\Nef^e(X):=\Nef(X)\cap\Eff(X),\qquad
\Nef^+(X):=\operatorname{conv}\bigl(\Nef(X)\cap N^1(X)_\QQ\bigr).
\]
The inclusion $\Nef^e(X)\subseteq\Nef^+(X)$ is noted in~\cite[Remark~1.4]{markman2015proof}.

A reduced irreducible divisor $E$ with $q_X(E)<0$ is called prime exceptional. For brevity, we call a ray spanned by a prime exceptional divisor a prime exceptional ray.

The Morrison--Kawamata cone conjecture predicts a rational polyhedral fundamental domain for the action of $\Aut(X)$ on $\Nef^e(X)$. For projective simple hyperk\"ahler manifolds with $\rho(X)\ge3$ and $b_2(X)\neq5$, Amerik and Verbitsky~\cite[Theorem~5.6]{AV17} proved that $\Aut(X)$ admits a finite polyhedral fundamental domain on $\Nef^+(X)$.

\subsection{Beauville involution}\label{subsec:Beauville_involution}

\begin{defn}[Beauville involution]
Let $S$ be a K3 surface and $X=S^{[2]}$. A Beauville involution on $X$ is a non-symplectic involution $\iota$ whose invariant lattice $H^2(X,\ZZ)^{\iota^*}$ has rank one and is generated by an ample class $D=H-e$ of square $2$, where $H$ is a very ample divisor of square $4$ on $S$.
Its action on $H^2(X,\ZZ)$ is
\[
\iota^*(x)=-x+\frac{2(x,D)}{q(D)}D,
\]
so $\iota^*$ fixes $D$ and acts as $-1$ on $D^\perp$.
\end{defn}

\begin{rmk}
The very ample divisor $H$ embeds $S$ as a quartic surface in $\PP^3$. If this quartic contains no line, the Beauville involution admits the following geometric realization: for two general points $p,q\in S$, the line $\overline{pq}\subset\PP^3$ meets $S$ in two further points $p',q'$, and the correspondence $(p,q)\mapsto(p',q')$ defines a regular automorphism of $S^{[2]}$~\cite{Bea83a}.
\end{rmk}

\subsection{Nef and Mori cones via the Mukai lattice}
We use Bayer--Macr\`i's Mukai lattice description of the nef and Mori cones of moduli spaces on K3 surfaces, based on Bridgeland wall-crossing~\cite{Bri07,Bri08,BM14a,BM14b}. The Hilbert scheme $S^{[n]}$ is the moduli space of ideal sheaves $I_Z$, where $Z\subset S$ is a length-$n$ subscheme. Set $X=S^{[n]}$ and $v:=v(I_Z)=(1,0,1-n)$. Write
\[
H^*_{\mathrm{alg}}(S,\ZZ) = H^0(S,\ZZ)\oplus\NS(S)\oplus H^4(S,\ZZ)
\]
for the algebraic Mukai lattice. For Mukai vectors $u=(r,c,s)$ and $u'=(r',c',s')$, we use the pairing
\[
(u,u')=c\cdot c'-rs'-r's.
\]
A class is called spherical if it has square $-2$ and is called isotropic if it has square $0$. Set $v^\perp:=\{u\in H^*_{\mathrm{alg}}(S,\ZZ)\mid (u,v)=0\}$, the orthogonal complement of $v$ in the algebraic Mukai lattice.

The Mukai morphism gives an isometry
\[
\theta_v\colon v^\perp\longrightarrow\NS(X),
\]
with
\[
\theta_v(0,-L,0)=L^{[n]}\quad\text{for }L\in\NS(S), \qquad \theta_v(1,0,n-1)=-e.
\]
We use the same notation for its $\QQ$-linear extension $v^\perp\otimes\QQ\to N^1(X)_\QQ$.

Let
\[
\theta_v^\vee\colon \bigl(H^*_{\mathrm{alg}}(S,\ZZ)/\ZZ v\bigr)\otimes\QQ \longrightarrow N_1(X)_\QQ
\]
denote the dual Mukai morphism, characterized by
\begin{equation}\label{eq:dual-mukai}
D\cdot\theta_v^\vee(w)=\bigl(\theta_v^{-1}(D),w\bigr) \qquad (D\in N^1(X)_\QQ).
\end{equation}
This is well-defined modulo $\QQ v$ because $\theta_v^{-1}(D)\in v^\perp\otimes\QQ$.
The Hilbert--Chow curve class satisfies $C_0=\theta_v^\vee(0,0,-1)$.

\begin{lem}\label{lem:pencil-mukai}
Let $C\subset S$ be a smooth curve admitting a $g^1_n$, and let $C_{[n]}\subset S^{[n]}$ be the associated curve. Then
\[
C_{[n]}=\theta_v^\vee\bigl(v(\Ocal_S(-C))\bigr).
\]
\end{lem}
\begin{proof}
By adjunction, $v(\Ocal_S(-C))=(1,-C,g(C))$. For every $D\in N^1(S)_\QQ$, equation~\eqref{eq:dual-mukai} and the formulas above for $\theta_v$ give
\[
\begin{aligned}
D^{[n]}\cdot\theta_v^\vee\bigl(v(\Ocal_S(-C))\bigr)
&=\bigl((0,-D,0),(1,-C,g(C))\bigr)=D\cdot C,\\
e\cdot\theta_v^\vee\bigl(v(\Ocal_S(-C))\bigr)
&=\bigl((-1,0,1-n),(1,-C,g(C))\bigr)=n+g(C)-1.
\end{aligned}
\]
These are precisely the intersection numbers of $C_{[n]}$. Since the classes $D^{[n]}$, together with $e$, span $N^1(X)_\QQ$, the two curve classes are equal.
\end{proof}

\begin{thm}[Bayer--Macr\`i {\cite[Theorems~12.1 and~12.2]{BM14b}}]\label{thm:BM-nef-mori}
The nef cone of $X$ is the chamber in $\overline{\Pos(X)}$ cut out by the hyperplanes
\[
\theta_v\bigl(v^\perp\cap w^\perp\bigr),
\]
where $0\ne w\in H^*_{\mathrm{alg}}(S,\ZZ)$ satisfies
\[
w^2\ge-2, \qquad 0\le (v,w)\le n-1.
\]
Dually, for an ample divisor $H$ on $X$, the BBF identification realizes $\NEbar(X)$ as the cone generated by $\overline{\Pos(X)}$ together with the classes $\theta_v^\vee(w)$ satisfying
\[
w^2\ge-2, \qquad |(v,w)|\le n-1, \qquad \theta_v^\vee(w)\cdot H>0.
\]
\end{thm}

For a primitive class $\rho\in H^2(X,\ZZ)$, its divisibility is the positive generator of the ideal $(\rho,H^2(X,\ZZ))\subset\ZZ$. We record the following facts for Hilbert squares.

\begin{lem}[{\cite[Proposition~2.12]{Mon15}}]\label{lem:hilbert-square-wall-types}
Let $X=S^{[2]}$. Every wall of $\Nef(X)$ is of the form $\rho^\perp$ for a primitive class $\rho\in\NS(X)$ such that either $q(\rho)=-2$, or $q(\rho)=-10$ and $\rho$ has divisibility two in $H^2(X,\ZZ)$.
\end{lem}

\begin{lem}\label{lem:minus-two-wall-exceptional}
Let $X=S^{[2]}$. If a wall of $\Nef(X)$ is orthogonal to a primitive class $\rho\in\NS(X)$ with $q(\rho)=-2$, then, after changing its sign if necessary, $\rho$ spans a prime exceptional ray.
\end{lem}
\begin{proof}
Let $A$ be an ample divisor on $X$. After replacing $\rho$ by $-\rho$ if necessary, assume $(\rho,A)>0$. Set $s:=\theta_v^{-1}(\rho)\in v^\perp$. Since $\theta_v$ is an isometry, $s^2=-2$, so $s$ is spherical. By Bayer--Macr\`i~\cite[Theorem~5.7(a)]{BM14b}, the corresponding wall induces a divisorial contraction. By~\cite[Theorem~12.4]{BM14b}, the class of the associated exceptional divisor is a positive multiple of $\theta_v(s)=\rho$. Hence $\rho$ spans a prime exceptional ray.
\end{proof}

\begin{lem}\label{lem:exceptional-rays-minus-two}
Let $X=S^{[2]}$. Every prime exceptional ray of $X$ has a primitive generator of BBF square $-2$.
\end{lem}
\begin{proof}
The Mukai vector is $v=(1,0,-1)$, so $v^2=2$. By Bayer--Macr\`i~\cite[Theorem~12.4]{BM14b}, a prime exceptional ray comes either from a spherical class $s\in v^\perp$ with $s^2=-2$, or from an isotropic class $w$ with $k:=(v,w)\in\{1,2\}$. The first case gives a primitive generator of square $-2$. In the second case, set $u:=(2w-kv)/k$. Since $k\in\{1,2\}$, this is an integral class, and $(u,v)=0$. Moreover,
\[
u^2=\frac{4w^2-4k(v,w)+k^2v^2}{k^2}=-2.
\]
Thus $u$ is primitive. The exceptional divisor class in~\cite[Theorem~12.4]{BM14b} is a positive multiple of
\[
\theta_v(2w-kv)=k\theta_v(u),
\]
so $\theta_v(u)$ is the primitive generator of the corresponding exceptional ray.
\end{proof}

\subsection{Generalized Fibonacci sequences}\label{subsec:general-Fibonacci}
Fix $a\ge1$ and define the generalized Fibonacci sequence $(a_m)_{m\in\ZZ}$ by $a_0=0$, $a_1=1$, and $a_{m+2}=a\,a_{m+1}+a_m$.
For $a=1$, this is the classical Fibonacci sequence.

Let $\alpha$ and $\beta$ denote the roots of the equation $r^2-ar-1=0$:
\[
\alpha=\frac{a-\sqrt{a^2+4}}{2},\qquad \beta=\frac{a+\sqrt{a^2+4}}{2}.
\]
They satisfy $\alpha+\beta=a$ and $\alpha\beta=-1$.

\begin{prop}\label{prop:general-Fibonacci-identities}
The sequence $(a_m)$ satisfies the following properties.

\begin{enumerate}
\item[(1)] (Negative indices) For $m\ge1$, $a_{-m}=(-1)^{m+1}a_m$.

\item[(2)] (Addition and subtraction formulas) For all $m,n\in\ZZ$,
\[
a_{m+n}=a_m a_{n+1}+a_{m-1}a_n,\qquad
a_{m-n}=(-1)^n\bigl(a_m a_{n-1}-a_{m-1}a_n\bigr).
\]

\item[(3)] (d'Ocagne identity) For all $m,n\in\ZZ$, $a_m a_{n+1}-a_{m+1}a_n=(-1)^n a_{m-n}$.

\item[(4)] (Binet formula) For every $m\in\ZZ$,
\[
a_m=\frac{\beta^m-\alpha^m}{\beta-\alpha}.
\]
\end{enumerate}
\end{prop}
\begin{proof}
The first identity follows by applying the recurrence backwards. For fixed $m$, both sides of the addition formula in (2) satisfy the same recurrence in $n$ and agree for $n=0,1$, so the formula holds for all $m,n\in\ZZ$. Replacing $n$ by $-n$ and using the first identity gives the subtraction formula. The d'Ocagne identity follows from
\[
a_m a_{n+1}-a_{m+1}a_n =a_m a_{n-1}-a_{m-1}a_n =(-1)^n a_{m-n},
\]
where the last equality is the subtraction formula. Finally, since $\alpha$ and $\beta$ are the roots of $r^2-ar-1=0$, the sequence $(\beta^m-\alpha^m)/(\beta-\alpha)$ satisfies the same recurrence and initial conditions as $a_m$, proving (4).
\end{proof}

\section{Nef and effective cones of Hilbert squares}\label{sec:Nef_cone}
Let $a\ge1$, set $S=S_a$ and $X=S^{[2]}$, and write classes in the basis $(h_1,e,h_2)$ of $\NS(X)$. Then $q(x,y,z)=4x^2+4axz-4z^2-2y^2$. For $k\in\ZZ$, set $F_k:=a_{2k-1}h_1+a_{2k}h_2$. The identities in Proposition~\ref{prop:general-Fibonacci-identities} give
\[
F_k^2=4,\qquad F_k\cdot F_{k+1}=2(a^2+2).
\]

The recurrence together with the negative-index formula gives $F_k\cdot h_1>0$, so $F_k$ lies in the component of the positive cone containing $h_1$. Since $\gcd(a_{m+1},a_m)=\gcd(a_m,a_{m-1})$ and $\gcd(a_1,a_0)=1$, consecutive terms of the sequence $(a_m)$ are coprime. Hence $F_k$ is primitive. Together with the absence of nonzero isotropic and $(-2)$-classes in $\NS(S_a)$, this excludes the exceptional cases in Saint-Donat's criterion~\cite{SD74}, so $F_k$ is very ample. The quartic model defined by $F_k$ contains no lines, since a line would be a rational curve and hence have square $-2$ by adjunction. Thus the associated Beauville involution $\iota_k$ is regular.

In the basis $(h_1,e,h_2)$,
\[
\iota_0^* =
\begin{pmatrix}
3 & 2 & 2a \\
-4 & -3 & -2a \\
0 & 0 & -1
\end{pmatrix},
\qquad
P:=
\begin{pmatrix}
1 & 0 & a \\
0 & 1 & 0 \\
a & 0 & a^2+1
\end{pmatrix}.
\]
The matrix $P$ fixes $e$ and preserves the BBF form. Since $F_k=P^kF_0$, the reflection formula gives
\[
\iota_k^*=P^k\iota_0^*P^{-k}.
\]

Recall that $\alpha=(a-\sqrt{a^2+4})/2$ and $\beta=(a+\sqrt{a^2+4})/2$.

The following numerical lemma provides the exclusion needed to classify the prime exceptional rays.

\begin{lem}\label{lem:interval-covering}
For $a=1,2$, let $\xi=L+ye\in\NS(X)$, where $L=xh_1+zh_2$. Suppose that $q(\xi)=-2$, $x>0$, $y<0$, and $F_k\cdot L\ge-4y$ for every $k\in\ZZ$. Then no such $\xi$ exists.
\end{lem}
\begin{proof}
Set $r:=z/x$. From $q(\xi)=-2$, we have
\[
1+ar-r^2=\frac{y^2-1}{2x^2}\ge0.
\]
Hence $r\in[\alpha,\beta]$. Since $r\in\QQ$ while $\alpha$ and $\beta$ are irrational, in fact $r\in(\alpha,\beta)$. Since $F_0=h_1$, the case $k=0$ gives $2x+az\ge-2y>0$, which, together with $q(\xi)=-2$, yields $(a^2+8)r^2-4ar-4>0$. Hence
\[
r\notin I_0:= \Biggl[ \frac{2a-2\sqrt{2a^2+8}}{a^2+8}, \frac{2a+2\sqrt{2a^2+8}}{a^2+8} \Biggr].
\]
Since $F_k=P^kF_0$ and $P$ preserves the BBF form and fixes $e$, we have $F_k\cdot L=F_0\cdot P^{-k}L$. The slope of $P^{-k}L$ is well-defined for every $k$, since otherwise $q(P^{-k}\xi)=-2$ and $F_0\cdot P^{-k}L\ge-4y>0$ give a contradiction. Thus the $k$-th inequality is the $k=0$ inequality applied to $P^{-k}\xi$. On slopes, $P$ acts by $f(r)=(a+(a^2+1)r)/(1+ar)$, so $f^{-k}(r)\notin I_0$, or equivalently $r\notin I_k:=f^k(I_0)$, for every $k$.

Set $g(t):=(a^2+8)t^2-4at-4$. Then $g(0)=-4$, while $g(\alpha)=(a^2+4)\alpha^2>0$ and $g(\beta)=(a^2+4)\beta^2>0$. Hence the two roots of $g$ lie in $(\alpha,0)$ and $(0,\beta)$, respectively, so $I_0\subset(\alpha,\beta)$. The map $\psi(r):=(r-\alpha)/(\beta-r)$ identifies $(\alpha,\beta)$ with $(0,\infty)$ and satisfies $\psi(f(r))=\beta^4\psi(r)$. Write $\psi(I_0)=[A,B]$. A direct calculation gives
\[
\frac{B}{A}=17+12\sqrt2.
\]
Hence $\psi(I_k)=\beta^{4k}[A,B]$. For $a=1,2$, we have $\beta^4\le17+12\sqrt2=B/A$. Thus consecutive intervals overlap, while their endpoints tend to $0$ and $\infty$ as $k\to-\infty$ and $k\to\infty$, respectively. Hence they cover $(0,\infty)$, contradicting $r\notin I_k$ for every $k$.
\end{proof}

\begin{lem}\label{lem:prime-exceptional-orbits}
For $a=1,2$, let $X=S_a^{[2]}$. Then the prime exceptional rays of $X$ are precisely the rays generated by the classes in $\Aut(X)\cdot e$.
\end{lem}
\begin{proof}
Let $G\subset\Aut(X)$ be the subgroup generated by the Beauville involutions $\iota_k$. Since $B=2e$ is prime exceptional and automorphisms preserve prime exceptional divisors, every ray generated by a class in $G\cdot e$ is prime exceptional. Suppose that there is a prime exceptional ray not generated by a class in $G\cdot e$. Let $\xi=(x,y,z)$ be its primitive generator and let $V$ be a prime divisor spanning the ray. By Lemma~\ref{lem:exceptional-rays-minus-two}, $\xi^2=-2$. Since $V\ne B$, \cite[Proposition~4.2(ii)]{boucksom2004divisorial} gives $(2e,[V])\ge0$. As $[V]$ is a positive multiple of $\xi$, we obtain $(e,\xi)\ge0$. If equality held, then $y=0$, and $\xi^2=-2$ would give $-2=4(x^2+axz-z^2)$, a contradiction. Hence $(e,\xi)>0$, so $y<0$.

Set $L:=xh_1+zh_2$. Since $B$ is the only divisor contracted by the Hilbert--Chow morphism and $V\ne B$, the image of $V$ is a divisor on $S_a^{(2)}$. Since $h_1^{[2]}$ is the pullback of an ample class on $S_a^{(2)}$, the projection formula gives $V\cdot(h_1^{[2]})^3>0$. The Fujiki relation then gives $([V],h_1)>0$, hence $(\xi,h_1)=(L,h_1)>0$. From $\xi^2=-2$, we have $L^2=2(y^2-1)\ge0$. As $S_a$ has no nonzero isotropic classes, $L^2>0$. Thus $L$ lies in the positive component containing $h_1$. Since $q(L)>0$ forces $x\ne0$, the two components are distinguished by the sign of $x$, so $x>0$.

Among the prime exceptional rays not generated by classes in $G\cdot e$, choose one minimizing the positive integer $(e,\xi)=-2y$. For every $k$, the divisor $\iota_k^*V$ is prime exceptional, and its ray is again not generated by a class in $G\cdot e$, since $G\cdot e$ is $G$-invariant. Minimality therefore gives $(e,\iota_k^*\xi)\ge(e,\xi)$. Since $\iota_k^*$ is the reflection in $F_k-e$, we have $(e,\iota_k^*\xi)-(e,\xi)=2(F_k\cdot L+4y)$, so $F_k\cdot L\ge-4y$ for every $k$. Lemma~\ref{lem:interval-covering} gives a contradiction. Hence every prime exceptional ray is generated by a class in $G\cdot e$.

Conversely, if $\varphi\in\Aut(X)$, then $\varphi^*B$ is prime exceptional, so $\varphi^*e$ spans a prime exceptional ray. Hence $\RR_{\ge0}\varphi^*e=\RR_{\ge0}g^*e$ for some $g\in G$. Thus the rays generated by $\Aut(X)\cdot e$ are precisely the prime exceptional rays.
\end{proof}

\begin{lem}\label{lem:no-rational-isotropic-cayley}
Let $X=S_1^{[2]}$. Then $N^1(X)_\QQ$ contains no nonzero isotropic class.
\end{lem}
\begin{proof}
After clearing denominators, a nonzero rational solution of $q(x,y,z)=0$ would give a primitive integral one. The equation $q(x,y,z)=0$ is equivalent to
\[
(2x+z)^2-2y^2=5z^2.
\]
Reducing modulo $5$, we obtain
\[
(2x+z)^2\equiv 2y^2\pmod 5.
\]
Since $2$ is a nonsquare modulo $5$, this congruence forces both sides to vanish modulo $5$. Hence $5\mid y$ and $5\mid(2x+z)$. It follows that both terms on the left-hand side of $(2x+z)^2-2y^2=5z^2$ are divisible by $25$, so $5\mid z$. Since $5\mid(2x+z)$ as well, we then obtain $5\mid x$, contradicting primitivity.
\end{proof}

\begin{thm}\label{thm:nef_cone_of_S^[2]}
Let $X=S_1^{[2]}$. The walls of the nef cone of $X$ are precisely the hyperplanes orthogonal to the classes in the orbit $\Aut(X)\cdot e$. Moreover,
\[
\Eff(X)=\operatorname{Cone}\bigl(\Aut(X)\cdot e\bigr).
\]
\end{thm}
\begin{proof}
By Lemma~\ref{lem:hilbert-square-wall-types}, every nef wall of $X$ is orthogonal either to a primitive integral class of square $-2$ or to one of square $-10$ and divisibility two. We first show that the second case cannot occur. If $q(x,y,z)=-10$, then $2(x^2+xz-z^2)-y^2=-5$. Reducing modulo $2$ shows that $y$ is odd, and reducing modulo $8$ then gives $x^2+xz-z^2\equiv2\pmod4$, whereas this quadratic form takes only the values $0,1,3$ modulo $4$. Thus every nef wall is orthogonal to a primitive integral class of square $-2$.

By Lemma~\ref{lem:minus-two-wall-exceptional}, the square $-2$ class defining such a wall may be chosen to span a prime exceptional ray. Lemma~\ref{lem:prime-exceptional-orbits} therefore shows that every nef wall is orthogonal to a class in $\Aut(X)\cdot e$. Conversely, $e^\perp$ is the nef wall corresponding to the Hilbert--Chow contraction, and automorphisms preserve the nef cone, so every hyperplane orthogonal to a class in $\Aut(X)\cdot e$ is a nef wall. This proves the description of the nef walls.

To describe the effective cone, set $C:=\operatorname{Cone}(\Aut(X)\cdot e)$. Since $\rho(X)=3$ and $B$ is prime exceptional, it follows from \cite[Theorem~1.2]{denisi2022pseudo} and Lemma~\ref{lem:prime-exceptional-orbits} that $\overline{\Eff}(X)=\overline C$. Bayer--Macr\`i~\cite[Theorem~12.4]{BM14b} also gives
\[
\Eff(X)=\operatorname{Cone}\bigl(C,\,\overline{\Pos(X)}\cap N^1(X)_\QQ\bigr).
\]
In particular, every rational class in $\Pos(X)$ lies in $\overline C$. Since the rational classes are dense in $\Pos(X)$, we obtain $\Pos(X)\subseteq\overline C$. As $\Pos(X)$ is open, $\Pos(X)\subseteq\Int(\overline C)$. Since $C$ is convex, $\Int(\overline C)=\Int(C)$, and therefore $\Pos(X)\subset C$. Lemma~\ref{lem:no-rational-isotropic-cayley} excludes nonzero rational isotropic classes, so the additional rational generators above are already contained in $C$. Therefore $\Eff(X)=C$.
\end{proof}

\begin{thm}\label{thm:nef-cone-a2}
Let $X=S_2^{[2]}$. The walls of $\Nef(X)$ are precisely the hyperplanes orthogonal to the classes in $\Aut(X)\cdot e$. Moreover,
\[
\begin{gathered}
\overline{\Eff}(X)
=\overline{\operatorname{Cone}\bigl(\Aut(X)\cdot e\bigr)},\\
\Eff(X)
=\operatorname{Cone}\left(\Aut(X)\cdot e,\,\left\{D\in N^1(X)_\QQ\;\middle|\;q(D)=0,\ D\in\overline{\Pos(X)}\right\}\right).
\end{gathered}
\]
\end{thm}
\begin{proof}
By Lemma~\ref{lem:hilbert-square-wall-types}, it remains to exclude the case $q(\rho)=-10$ with divisibility two. Suppose that a nef wall is of the form $\rho^\perp$, where $\rho=xh_1+ye+zh_2$ is primitive, $q(\rho)=-10$, and $\rho$ has divisibility two. Then $(xh_1+zh_2,\eta)\in2\ZZ$ for every $\eta\in H^2(S_2,\ZZ)$, so unimodularity of the K3 lattice implies $(xh_1+zh_2)/2\in H^2(S_2,\ZZ)$. Since this class is of type $(1,1)$, it lies in $\NS(S_2)$. Hence $x$ and $z$ are even. But $2x^2+4xz-2z^2-y^2=-5$, which modulo $8$ gives $y^2\equiv5$, a contradiction.

Thus every nef wall is orthogonal to a primitive integral class of square $-2$. The same argument as in the proof of Theorem~\ref{thm:nef_cone_of_S^[2]}, using Lemma~\ref{lem:minus-two-wall-exceptional} and Lemma~\ref{lem:prime-exceptional-orbits}, gives the stated description of the nef walls.

The effective cone argument in the same proof gives the stated formulas, except that rational isotropic classes on the boundary of $\overline{\Pos(X)}$ are not excluded here.
\end{proof}

For $a=1,2$, since $\iota_0^*e=2h_1-3e$ and $P$ fixes $e$, the relation $\iota_k^*=P^k\iota_0^*P^{-k}$ gives $\iota_k^*e=2F_k-3e$. By the Binet formula, the coefficients $a_{2k}$ are strictly increasing with $k$, so the classes $F_k$ are pairwise distinct. Since every $\iota_k^*e$ has $e$-coefficient $-3$, the classes $\iota_k^*e$ are pairwise nonproportional. By Theorems~\ref{thm:nef_cone_of_S^[2]} and~\ref{thm:nef-cone-a2}, their orthogonal hyperplanes are therefore distinct nef walls. Hence $\Nef(S_a^{[2]})$ is not polyhedral for $a=1,2$; in particular, $S_a^{[2]}$ is not a Mori dream space.

\section{Fundamental domain}\label{sec:FD}
Let $X=S_1^{[2]}$. We now explicitly construct a rational polyhedral fundamental domain for the action of $\Aut(X)$ on the nef effective cone. The corresponding existence statement for $\Nef^+(X)$ is recalled in Section~\ref{subsec:Cone_conjecture}.

By Lemma~\ref{lem:no-rational-isotropic-cayley}, $N^1(X)_\QQ$ has no nonzero isotropic class. Hence every nonzero rational nef class has positive square, and by convexity of the positive cone every nonzero class in $\Nef^+(X)$ lies in $\Pos(X)$. Therefore
\[
\Nef^e(X)\subseteq\Nef^+(X)
\subseteq \{0\}\cup\bigl(\Nef(X)\cap\Pos(X)\bigr),
\]
where the first inclusion is recalled in Section~\ref{subsec:Cone_conjecture}. Conversely, every nef class of positive square is big and hence effective. Thus
\[
\Nef^+(X)=\Nef^e(X)
=\{0\}\cup\bigl(\Nef(X)\cap\Pos(X)\bigr).
\]

For a class $p=xh_1+ye+zh_2$, set $\ell_1:=x-z$, $\ell_2:=z$, and $\ell_3:=3x+2y$. The classes $e$, $\iota_0^*e=(2,-3,0)$, and $\iota_1^*e=(2,-3,2)$ give, respectively, the following three inequalities on the nef cone:
\[
3\ell_1+3\ell_2-\ell_3\ge0, \qquad -\ell_1+3\ell_2+3\ell_3\ge0, \qquad 3\ell_1-\ell_2+3\ell_3\ge0.
\]

A nef class cannot have two negative $\ell$-coordinates. Indeed, if $\ell_i,\ell_j<0$ for distinct $i,j$, and $k$ is the remaining index, the corresponding inequality gives $\ell_k\le3(\ell_i+\ell_j)<0$. Thus all three would be negative, contradicting the inequality $5(\ell_1+\ell_2+\ell_3)\ge0$ obtained by summing the three inequalities. Hence the nef effective cone is partitioned into the four regions
\[
\begin{cases}
(1)\ \ell_1<0,\\[2pt]
(2)\ \ell_2<0,\\[2pt]
(3)\ \ell_3<0,\\[2pt]
(4)\ \ell_1,\ell_2,\ell_3\ge0.
\end{cases}
\]
We denote region $(4)$ by $\Pi$.
\begin{thm}\label{thm:fund_domain}
The cone $\Pi$ is a rational polyhedral fundamental domain for the action of $\Aut(X)$ on $\Nef^e(X)$.
\end{thm}

\begin{proof}
Define the following transformations of $N^1(X)$:
\[
A_1:=\iota_1^*,\qquad A_2:=\iota_0^*,\qquad A_3:=\iota_1^*\iota_2^*\iota_1^*.
\]
We first describe $\Pi$ explicitly. Let $C$ be the cone defined by the three nef inequalities above together with $\ell_1,\ell_2,\ell_3\ge0$. Since $\Pi$ is region $(4)$, every class in $\Pi$ satisfies these six inequalities, so $\Pi\subseteq C$. The six defining hyperplanes of $C$ are
\[
x-z=0,\quad z=0,\quad 3x+2y=0,\quad 8x+6y+4z=0,\quad 12x+6y-4z=0,\quad y=0.
\]
A direct calculation shows that $C$ is generated by the six rays
\[
h_1,\quad h_1+h_2,\quad 3h_1-4e,\quad 3h_1-4e+3h_2,\quad 4h_1-6e+h_2,\quad 4h_1-6e+3h_2.
\]
All six rays lie in $\Nef^e(X)$. Indeed, $h_1=F_0$ and $h_1+h_2=F_1$ are nef and effective. The classes $3h_1-4e=A_2h_1$ and $3h_1-4e+3h_2=A_1(h_1+h_2)$ are their images under automorphisms. The remaining two are positive combinations of such classes, since $3(4h_1-6e+h_2)=A_2h_1+A_3A_1(h_1+h_2)$ and $3(4h_1-6e+3h_2)=A_1(h_1+h_2)+A_3A_2h_1$. Hence $C\subseteq\Nef^e(X)$. Since $\ell_1,\ell_2,\ell_3\ge0$ on $C$, we also have $C\subseteq\Pi$. Therefore $C=\Pi$, and $\Pi$ is rational polyhedral.

We next prove that the translates of $\Pi$ cover $\Nef^e(X)$. By the description of $\Nef^e(X)$ above and the fact that $0\in\Pi$, it remains to consider a nef class $p=(x,y,z)$ with $q(p)>0$. If $p\notin\Pi$, then $p$ lies in a unique region $(i)$ with $i\in\{1,2,3\}$. In the coordinates $(\ell_1,\ell_2,\ell_3)$, the transformations $A_i$ act by
\[
\begin{aligned}
A_1&:(\ell_1,\ell_2,\ell_3)\longmapsto(-\ell_1,3\ell_1+\ell_3,3\ell_1+\ell_2),\\
A_2&:(\ell_1,\ell_2,\ell_3)\longmapsto(3\ell_2+\ell_3,-\ell_2,\ell_1+3\ell_2),\\
A_3&:(\ell_1,\ell_2,\ell_3)\longmapsto(\ell_2+3\ell_3,\ell_1+3\ell_3,-\ell_3).
\end{aligned}
\]
In particular, $A_i$ reverses the sign of $\ell_i$. We therefore apply $A_i$ whenever the current class lies in region $(i)$.

The function $f(x,y,z):=2x+y=\frac12(\ell_1+\ell_2+\ell_3)$ satisfies
\[
f(A_1p)-f(p)=2\ell_1,\qquad f(A_2p)-f(p)=2\ell_2,\qquad f(A_3p)-f(p)=2\ell_3,
\]
so $f$ strictly decreases at each step of the reduction.

We claim that $f(p)>0$ for every nef class $p$ of positive square. Summing the three nef inequalities gives $5(\ell_1+\ell_2+\ell_3)\ge0$, so $f(p)\ge0$. If $f(p)=0$, then the three nonnegative left-hand sides of the nef inequalities have sum zero, so they all vanish. Hence $\ell_1=\ell_2=\ell_3=0$, contradicting $q(p)>0$. Thus $f(p)>0$, so $f$ remains positive throughout the reduction.

Suppose that the reduction does not terminate, and set $p_0=p$. Write $p_{m+1}=A_{i_m}p_m$, where $p_m$ lies in region $(i_m)$. Since the positive sequence $f(p_m)$ is strictly decreasing, it converges. Set
\[
\delta_m:=f(p_m)-f(p_{m+1})
=-2\ell_{i_m}(p_m)>0.
\]
Then $\delta_m\to0$. Since $A_{i_m}$ reverses the sign of $\ell_{i_m}$, we have $i_{m+1}\ne i_m$.

Fix $m$, set $i=i_m$ and $j=i_{m+1}$, and let $k$ be the remaining index. Since $p_m$ lies in region $(i)$, we have $\ell_j(p_m),\ell_k(p_m)\ge0$. Moreover,
\[
-\frac{\delta_{m+1}}2
=\ell_j(p_{m+1})
=3\ell_i(p_m)+\ell_k(p_m),
\]
so
\[
0\le \ell_k(p_m)
=\frac{3\delta_m-\delta_{m+1}}2
\le \frac32\delta_m.
\]
The nef inequality with coefficient $-1$ on $\ell_j$ then gives
\[
0\le \ell_j(p_m)
\le 3\bigl(\ell_i(p_m)+\ell_k(p_m)\bigr)
\le 3\delta_m.
\]
Together with $\lvert\ell_i(p_m)\rvert=\delta_m/2$, this shows that all three $\ell$-coordinates of $p_m$ tend to zero. Hence $q(p_m)\to0$, contradicting $q(p_m)=q(p_0)>0$, since the transformations $A_i$ preserve the BBF form. Thus the reduction terminates after finitely many steps in region $(4)=\Pi$. Therefore every class in $\Nef^e(X)$ can be carried into $\Pi$ by an automorphism, and
\[
\Nef^e(X)=\bigcup_{\varphi\in\Aut(X)}\varphi^*(\Pi).
\]

It remains to prove that distinct translates have disjoint interiors. Fix $i\in\{1,2,3\}$. If $p$ lies in the interior of region $(j)$ with $j\ne i$, then $\ell_i(p)>0$, while the formulas above give $\ell_i(A_i p)=-\ell_i(p)<0$. Since $A_i$ is invertible and preserves $\Nef^e(X)$, it preserves its interior. Hence $A_i p$ lies in the interior of region $(i)$. In particular, $A_i(\Int\Pi)$ lies in the interior of region $(i)$.

By~\cite{Lee25}, $\Aut(X)=\langle\iota_0,\iota_1,\iota_2\rangle \cong\ZZ_2*\ZZ_2*\ZZ_2$. Replacing $\iota_2$ by its conjugate $\iota_1\iota_2\iota_1$ gives another free generating set. Indeed, the map fixing $\iota_0,\iota_1$ and sending $\iota_2$ to $\iota_1\iota_2\iota_1$ is an automorphism of the free product.

Hence, for every nonidentity automorphism $\varphi\in\Aut(X)$, we may write $(\varphi^*)^{-1}=A_{i_k}\cdots A_{i_1}$ as a nonempty reduced word, so $i_{r+1}\ne i_r$ for $1\le r<k$. Starting from $\Int(\Pi)$ and applying the factors from right to left, the successive images lie in the interiors of regions $(i_1),\ldots,(i_k)$. Thus $(\varphi^*)^{-1}\Int(\Pi)$ lies in the interior of region $(i_k)$, which is disjoint from $\Int(\Pi)$, and hence
\[
(\varphi^*)^{-1}\Int(\Pi)\cap\Int(\Pi)=\varnothing
\]
for every nonidentity $\varphi\in\Aut(X)$. Applying this to the automorphism relating any two distinct translates shows that their interiors are pairwise disjoint. Hence $\Pi$ is a fundamental domain.
\end{proof}

\begin{figure}[htbp]
 \centering
 \includegraphics[width=\textwidth]{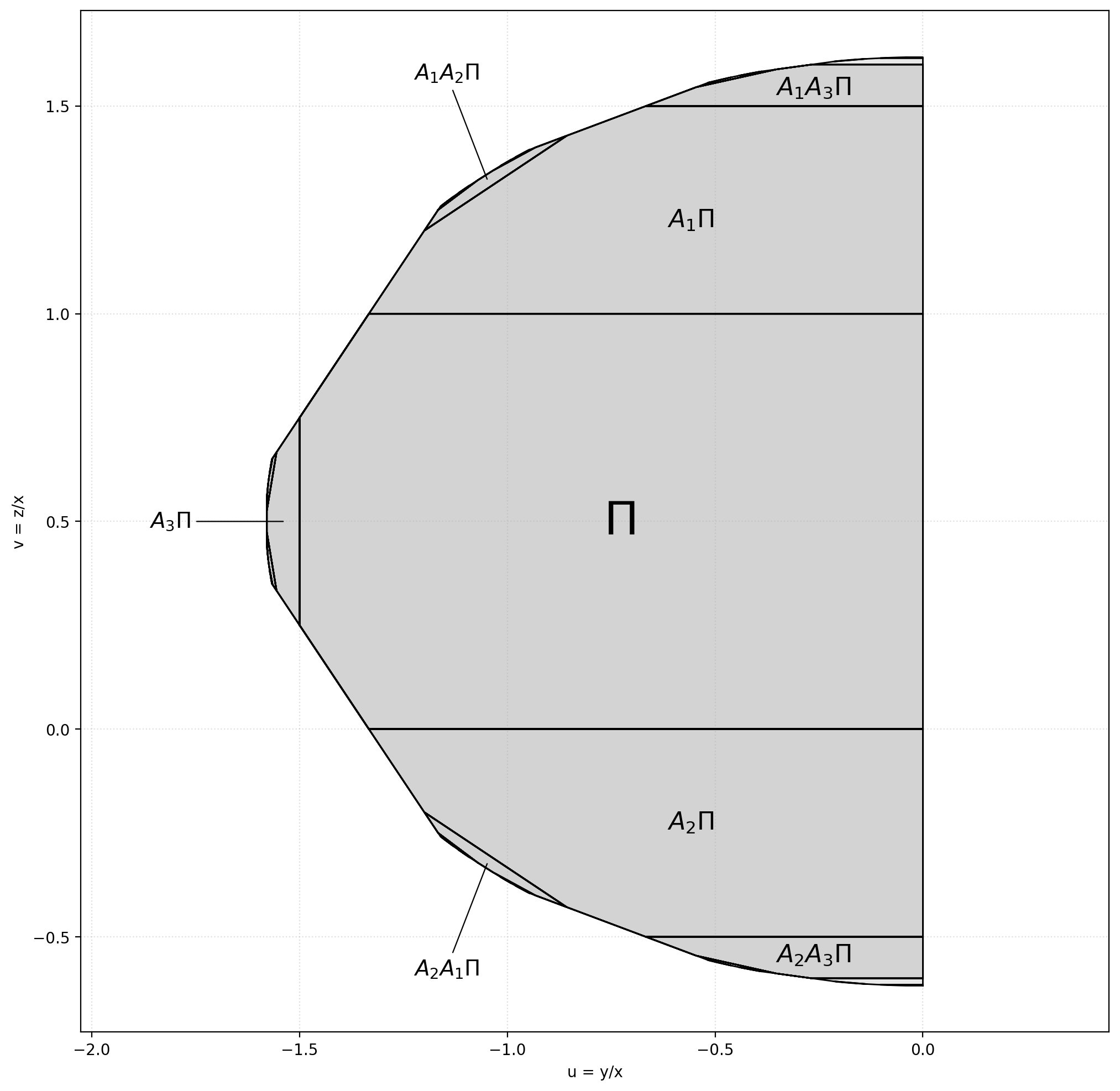}
 \caption{The $x=1$ slice of the nef cone of $S_1^{[2]}$ in $(y,z)$-coordinates, where classes are written as $xh_1+ye+zh_2$}
 \label{fig:FD}
\end{figure}

%\FloatBarrier

\section{Nef and Mori cones of higher Hilbert schemes}\label{sec:higher-2}
Starting from the classes $F_k$, we construct a chain of ample classes $E_{k,t}$ that will be used to describe the Mori cone and define the divisor classes $D_{k,t}$.

Let $a\ge1$ and $n\ge2$, and set $X=S_a^{[n]}$. Recall that $F_k:=a_{2k-1}h_1+a_{2k}h_2$ for $k\in\ZZ$. Set $G_k:=a_{2k}h_1+a_{2k+1}h_2=(F_{k+1}-F_k)/a$. By Proposition~\ref{prop:general-Fibonacci-identities}, $a_{2k-1}a_{2k+1}-a_{2k}^2=1$, so $(F_k,G_k)$ is an integral basis of $\NS(S_a)$.

For $k\in\ZZ$ and integers $0\le t\le a$, set
\[
E_{k,t}:=F_k+tG_k =\frac{a-t}{a}F_k+\frac{t}{a}F_{k+1}.
\]
Since $E_{k,0}=F_k$ and $E_{k,a}=F_{k+1}=E_{k+1,0}$, the classes $E_{k,t}$ form a chain joining $F_k$ to $F_{k+1}$. For $0\le t\le a-1$, the change-of-basis matrix from $(F_k,G_k)$ to $(E_{k,t},E_{k,t+1})$ has determinant $1$, so every such pair is an integral basis.

Using the intersection numbers of $F_k$ and $F_{k+1}$ computed in Section~\ref{sec:Nef_cone}, we have $E_{k,s}\cdot E_{k,t}=4+2a(s+t)-4st$ for $0\le s,t\le a$. In particular, $E_{k,t}^2=4(1+at-t^2)$. By adjunction, the arithmetic genus of a curve in $|E_{k,t}|$ is $\gamma_t:=1+\frac{E_{k,t}^2}{2}=3+2t(a-t)$, independent of $k$. For $0\le t\le a-1$, the same formula gives $E_{k,t}\cdot E_{k,t+1}=\gamma_t+\gamma_{t+1}$.

Since $E_{k,t}$ is a convex combination of the ample classes $F_k$ and $F_{k+1}$, it is ample. It is primitive because $(F_k,G_k)$ is an integral basis and $E_{k,t}=F_k+tG_k$. Together with the absence of isotropic and $(-2)$-classes in $\NS(S_a)$, this excludes the exceptional cases in Saint-Donat's criterion~\cite{SD74}, so every $E_{k,t}$ is very ample.

Set $N:=n-1$. If a smooth curve $C\in|E_{k,t}|$ admits a $g^1_n$, let $C_{[n]}\subset S_a^{[n]}$ be the associated curve. Although $C_{[n]}$ may depend on the choices of $C$ and the pencil, Lemma~\ref{lem:pencil-mukai} shows that its numerical class depends only on $E_{k,t}$; we denote this class by $(E_{k,t})_{[n]}\in N_1(X)_\QQ$. Then, for every $D\in\NS(S_a)$,
\[
D^{[n]}\cdot(E_{k,t})_{[n]}=D\cdot E_{k,t}, \qquad e\cdot(E_{k,t})_{[n]}=N+\gamma_t.
\]

To construct a class orthogonal to two consecutive curve classes of this form, fix $k\in\ZZ$ and $0\le t\le a-1$. Since $E_{k,t}$ and $E_{k,t+1}$ form a basis of $\NS(S_a)$, there is a unique class $P_{k,t}\in N^1(S_a)_\QQ\subset N^1(X)_\QQ$ satisfying $P_{k,t}\cdot E_{k,t}=N+\gamma_t$ and $P_{k,t}\cdot E_{k,t+1}=N+\gamma_{t+1}$. Set $D_{k,t}:=P_{k,t}-e$. By construction, whenever both corresponding pencils exist, $D_{k,t}\cdot(E_{k,t})_{[n]}=D_{k,t}\cdot(E_{k,t+1})_{[n]}=0$.

\begin{lem}\label{lem:mukai-wall-inequality}
For every $k\in\ZZ$ and $0\le t\le a-1$, the class $D_{k,t}$ is nef. If a curve class $\theta_v^\vee(w)\in N_1(X)_\QQ$ appearing in Theorem~\ref{thm:BM-nef-mori} is orthogonal to $D_{k,t}$, then
\[
w=v\bigl(\Ocal_{S_a}(-E_{k,t})\bigr) \qquad\text{or}\qquad w=v\bigl(\Ocal_{S_a}(-E_{k,t+1})\bigr).
\]
\end{lem}

\begin{proof}
Solving the two defining intersection equations for $P_{k,t}$ gives
\[
q(D_{k,t})
=
\frac{2\bigl((\gamma_t-N)^2+(\gamma_{t+1}-N)^2+4N\bigr)}
{(\gamma_t-\gamma_{t+1})^2+4(\gamma_t+\gamma_{t+1}-1)}>0.
\]
Moreover, the pullback divisor class $E_{k,t}^{[n]}$ lies in $\Pos(X)$ and, by the defining equation for $P_{k,t}$,
\[
\bigl(D_{k,t},E_{k,t}^{[n]}\bigr)=P_{k,t}\cdot E_{k,t}=N+\gamma_t>0.
\]
Hence $D_{k,t}\in\Pos(X)$.

Since $D_{k,t}\in\Pos(X)$, it pairs positively with every nonzero class in $\overline{\Pos(X)}$. Also, $D_{k,t}\cdot C_0=1$. Thus, by the dual description in Theorem~\ref{thm:BM-nef-mori}, it remains only to consider a class $\theta_v^\vee(w)$ not lying on the Hilbert--Chow ray $\RR_{\ge0}C_0$. For such a class, $w^2\ge-2$ and $|(v,w)|\le N$. Since $E_{k,t},E_{k,t+1}$ form an integral basis of $\NS(S_a)$, write $w=(r,-pE_{k,t}-qE_{k,t+1},s)$ for $p,q,r,s\in\ZZ$. With $v=(1,0,-N)$, set $m:=(v,w)$. Then $|m|\le N$, $s=rN-m$, and hence
\[
w=(r,-pE_{k,t}-qE_{k,t+1},rN-m).
\]

Since the cone of curves contracted by the Hilbert--Chow morphism is $\RR_{\ge0}C_0$, the pullback $(F_k+F_{k+1})^{[n]}$ of an ample class is strictly positive on $\theta_v^\vee(w)$. By~\eqref{eq:dual-mukai},
\[
\begin{aligned}
0<(F_k+F_{k+1})^{[n]}\cdot\theta_v^\vee(w)
&=(F_k+F_{k+1})\cdot(pE_{k,t}+qE_{k,t+1})\\
&=2(a^2+4)(p+q).
\end{aligned}
\]
Hence $p+q>0$.

Since $\theta_v^{-1}(D_{k,t})=(1,-P_{k,t},N)\in v^\perp\otimes\QQ$, equation~\eqref{eq:dual-mukai} and the defining equations for $P_{k,t}$ give $D_{k,t}\cdot\theta_v^\vee(w)=p(N+\gamma_t)+q(N+\gamma_{t+1})-2rN+m$. Combining this with the expression for $w^2$ gives
\begin{equation}\label{eq:mukai-wall-key}
(p+q)\bigl(D_{k,t}\cdot\theta_v^\vee(w)\bigr)
=
N(p+q-r)^2+m(p+q-r)+p^2+q^2+\frac{w^2}{2}.
\end{equation}
Set $d:=p+q-r$. Since $|m|\le N$ and $w^2\ge-2$, we have $md\ge-N|d|$ and $w^2/2\ge-1$. Hence
\[
(p+q)\bigl(D_{k,t}\cdot\theta_v^\vee(w)\bigr) \ge N|d|(|d|-1)+p^2+q^2-1.
\]
Since $d,p,q\in\ZZ$ and $p+q>0$, we have $|d|(|d|-1)\ge0$ and $p^2+q^2\ge1$. Hence the right-hand side is nonnegative, so $D_{k,t}\cdot\theta_v^\vee(w)\ge0$. Together with the positivity of $D_{k,t}$ on $\overline{\Pos(X)}\setminus\{0\}$ and the equality $D_{k,t}\cdot C_0=1$, this proves that $D_{k,t}$ is nef by Theorem~\ref{thm:BM-nef-mori}.

Suppose that $D_{k,t}\cdot\theta_v^\vee(w)=0$. Equality in the lower bound above gives $p^2+q^2=1$, $|d|\le1$, and $w^2=-2$. Since $p+q>0$, we have $p+q=1$, so $d=1-r$ and hence $r\in\{0,1,2\}$. If $r$ were even, then
\[
w^2=(pE_{k,t}+qE_{k,t+1})^2-2rs
\]
would be divisible by $4$, since every square in $\NS(S_a)$ is divisible by $4$, contradicting $w^2=-2$. Thus $r=1$ and $d=0$.

Therefore $w=(1,-E_{k,t},s)$ or $w=(1,-E_{k,t+1},s)$. Since $w^2=-2$, we obtain $s=1+E_{k,t}^2/2=\gamma_t$ in the first case and $s=\gamma_{t+1}$ in the second. As $v(\Ocal_{S_a}(-E_{k,u}))=(1,-E_{k,u},\gamma_u)$ for $u=t,t+1$, we conclude that
\[
w=v\bigl(\Ocal_{S_a}(-E_{k,t})\bigr) \qquad\text{or}\qquad w=v\bigl(\Ocal_{S_a}(-E_{k,t+1})\bigr). \qedhere
\]

\end{proof}

We now impose the numerical bound needed to realize the classes $(E_{k,t})_{[n]}$ by pencils on smooth curves. Assume $n\ge\lfloor a^2/4\rfloor+3$. For every $k\in\ZZ$ and $0\le t\le a-1$, choose a smooth curve $C_{k,t}\in|E_{k,t}|$. Since $\gamma_t=3+2t(a-t)\le 3+2\lfloor a^2/4\rfloor$ and $N=n-1\ge\lfloor a^2/4\rfloor+2$, we have $\rho(\gamma_t,1,n)=2N-\gamma_t\ge1$. The Brill--Noether existence theorem~\cite[Chapter~V, Theorem~1.1]{ACGH85} therefore gives a $g^1_n$ on $C_{k,t}$. Fix one such pencil and the associated curve in $X$ representing $(E_{k,t})_{[n]}$; for the endpoint $E_{k,a}=E_{k+1,0}$, use the pencil chosen for $C_{k+1,0}$.

For $k\in\ZZ$ and $0\le u\le a$, set $w_{k,u}:=v(\Ocal_{S_a}(-E_{k,u}))=(1,-E_{k,u},\gamma_u)$. Since $\gamma_u\le2N$, we have $w_{k,u}^2=-2$ and $|(v,w_{k,u})|=|N-\gamma_u|\le N$. By the choice of pencils above and Lemma~\ref{lem:pencil-mukai}, $(E_{k,u})_{[n]}=\theta_v^\vee(w_{k,u})$. Since $(E_{k,u})_{[n]}$ is effective, $\theta_v^\vee(w_{k,u})\cdot H>0$ for every ample class $H$. Thus these classes satisfy the conditions in Theorem~\ref{thm:BM-nef-mori}. Together with Lemma~\ref{lem:mukai-wall-inequality}, this shows that, for each $k,t$, the two consecutive pencil classes $(E_{k,t})_{[n]}$ and $(E_{k,t+1})_{[n]}$ are precisely the curve classes appearing in Theorem~\ref{thm:BM-nef-mori} that are orthogonal to $D_{k,t}$.

\begin{thm}\label{thm:nef-cone-large-n}
Let $a\ge1$ and $n\ge\lfloor a^2/4\rfloor+3$, and set $X=S_a^{[n]}$. Then
\[
\NEbar(X)
=
\overline{\operatorname{Cone}\bigl\langle
C_0,\,(E_{k,t})_{[n]}
\mid k\in\ZZ,\;0\le t\le a-1
\bigr\rangle}.
\]
Moreover,
\[
\Nef(X)=\overline{\operatorname{Cone}\bigl\langle
D_{k,t}\mid k\in\ZZ,\;0\le t\le a-1
\bigr\rangle}.
\]
\end{thm}

\begin{proof}
Using $E_{k,a}=E_{k+1,0}$, regard the classes $E_{k,t}$ as a single bi-infinite sequence. Write two consecutive classes as $E_{k,t}=xh_1+zh_2$ and $E_{k,t+1}=x'h_1+z'h_2$ with $0\le t\le a-1$. By the determinant computations above, $xz'-zx'=1$. The negative-index formula shows that $a_{2j-1}>0$ for every $j\in\ZZ$, so the $h_1$-coefficient of every $F_j$, and hence of every $E_{k,t}$, is positive. Thus $x,x'>0$, and it follows that $z/x<z'/x'$. Since $E_{k,t}$ lies between $F_k$ and $F_{k+1}$, the Binet formula in Proposition~\ref{prop:general-Fibonacci-identities} shows that these slopes tend to $\alpha$ and $\beta$ at the two ends of the sequence. For any $L=xh_1+zh_2\in\Pos(S_a)$, we have $x>0$ and $L^2=4x^2(1+a(z/x)-(z/x)^2)>0$, so $\alpha<z/x<\beta$. Hence $L$ lies in the cone generated by some consecutive pair $E_{k,t},E_{k,t+1}$.

The $h_1$-coefficients of $E_{k,t}$ tend to $+\infty$ at both ends. Set $u_-:=h_1+\alpha h_2$ and $u_+:=h_1+\beta h_2$; their rays are the two boundary rays of $\overline{\Pos(S_a)}$. We use the same notation for their images in $N^1(X)_\RR$ and, via the BBF identification, in $N_1(X)_\RR$.

Set
\[
\mathcal C:=\overline{\operatorname{Cone}\bigl\langle
C_0,\,(E_{k,t})_{[n]}
\mid k\in\ZZ,\;0\le t\le a-1
\bigr\rangle}.
\]
By construction, $\mathcal C\subseteq\NEbar(X)$.

We determine the facets of $\mathcal C$. Under the BBF identification, the pencil class $(E_{k,t})_{[n]}$ corresponds to
\[
E_{k,t}-\frac{N+\gamma_t}{2N}e.
\]
Together with the slope limits established above and the finiteness of the range of $t$, this shows that the only accumulation rays of the pencil classes are $\RR_{\ge0}u_-$ and $\RR_{\ge0}u_+$. In particular, both rays lie in $\mathcal C$.

Fix $k\in\ZZ$ and $0\le t\le a-1$. Among the pencil rays, the nef class $D_{k,t}$ is orthogonal precisely to $(E_{k,t})_{[n]}$ and $(E_{k,t+1})_{[n]}$, while $D_{k,t}\cdot C_0=1$. Since $D_{k,t}\in\Pos(X)$, it also pairs positively with $u_-$ and $u_+$. Since $u_-$ and $u_+$ are the only accumulation rays and $D_{k,t}$ pairs positively with both, the face of $\mathcal C$ cut out by $D_{k,t}$ is generated by these two pencil rays. Hence
\[
D_{k,t}^{\perp}\cap\mathcal C
=
\operatorname{Cone}\bigl((E_{k,t})_{[n]},(E_{k,t+1})_{[n]}\bigr).
\]
Since $\rho(X)=3$ and the two pencil rays are distinct, this face is a facet.

The classes $u_-$ and $u_+$ satisfy $u_\pm\cdot C_0=u_\pm^2=0$, and each pairs positively with every pencil ray and with the other accumulation ray. Hence they support the facets
\[
u_\pm^\perp\cap\mathcal C
=
\operatorname{Cone}(C_0,u_\pm).
\]

To see that these are all the facets, take the two-dimensional slice of $\mathcal C$ defined by $H\cdot R=1$ for an ample class $H$. This slice is the closed convex hull of the normalized points on the generating rays. The strict ordering of the $E_{k,t}$-slopes shows that the consecutive pencil facets form a boundary chain whose two ends converge to $u_-$ and $u_+$. The two facets above join these accumulation rays to $C_0$, completing the boundary of the slice. Hence these are all the facets of $\mathcal C$.

Dualizing this facet description gives
\[
\mathcal C^\vee
=
\overline{\operatorname{Cone}\bigl\langle
D_{k,t},u_-,u_+
\mid k\in\ZZ,\;0\le t\le a-1
\bigr\rangle}.
\]

It remains to show that the two accumulation rays are already contained in the closure of the cone generated by the $D_{k,t}$. Solving the defining equations for $P_{k,0}$ in the basis $(F_k,G_k)$ gives
\[
P_{k,0}=\mu F_k+\nu G_k,\qquad
\mu=\frac{N+a^2-a+3}{a^2+4},\qquad
\nu=\frac{a(N-1)+4}{2(a^2+4)}.
\]
Since $G_k=(F_{k+1}-F_k)/a$,
\[
P_{k,0}=\left(\mu-\frac{\nu}{a}\right)F_k+\frac{\nu}{a}F_{k+1}.
\]
The explicit formulas give $\nu>0$ and $a\mu>\nu$; indeed, $2(a^2+4)(a\mu-\nu)=aN+2a^3-2a^2+7a-4>0$. Hence both coefficients are positive.

Since the rays of $F_k$ and $F_{k+1}$ converge to $\RR_{\ge0}u_-$ as $k\to-\infty$ and to $\RR_{\ge0}u_+$ as $k\to+\infty$, the same holds for the rays of $P_{k,0}$. As the $h_1$-coefficients of $F_k$ and $F_{k+1}$ tend to $+\infty$ at both ends, the $h_1$-coefficient of $P_{k,0}$ does as well. Since $D_{k,0}=P_{k,0}-e$ and its $e$-coefficient is fixed, the same projective limits hold for the rays of $D_{k,0}$. Thus $u_-$ and $u_+$ lie in the closure of the cone generated by the $D_{k,t}$, so
\[
\mathcal C^\vee
=
\overline{\operatorname{Cone}\bigl\langle
D_{k,t}\mid k\in\ZZ,\;0\le t\le a-1
\bigr\rangle}
=:K.
\]

By Lemma~\ref{lem:mukai-wall-inequality}, every $D_{k,t}$ is nef, and hence $K\subseteq\Nef(X)$. Duality between the nef and Mori cones therefore gives
\[
\mathcal C
\subseteq\NEbar(X)=\Nef(X)^\vee
\subseteq K^\vee
=\mathcal C^{\vee\vee}
=\mathcal C.
\]
Thus $\NEbar(X)=\mathcal C$, and dualizing once more gives $\Nef(X)=K$.
\end{proof}

By the facet description above, the rays $\RR_{\ge0}D_{k,t}$ form an infinite collection of distinct extremal rays of $\Nef(X)$. Hence the nef cone is not polyhedral; in particular, $X$ is not a Mori dream space.

\begin{exa}
For $a=1$, $E_{k,0}=F_k=a_{2k-1}h_1+a_{2k}h_2$. Hence, for $n\ge3$, the nef cone of $S_1^{[n]}$ is cut out by the hyperplanes $(F_k)_{[n]}^\perp$ and $C_0^\perp$. The nef extremal rays indexed by $k\in\ZZ$ are generated by
\[
D_{k,0} = \frac{n+2}{10}(F_k+F_{k+1})-e,
\]
while the two accumulation rays are generated by $u_-$ and $u_+$.
\end{exa}

\bibliographystyle{plain}
\bibliography{bib}
\end{document}